\documentclass[letterpaper, 10 pt, conference]{ieeeconf}  

\IEEEoverridecommandlockouts                              

\usepackage{cite}
\usepackage{amsmath}
\usepackage{amssymb}
\usepackage{balance}
\usepackage{graphicx}
\usepackage{subfigure}
\usepackage{epstopdf}
\usepackage{amsfonts}
\usepackage[all]{xy}
\usepackage{array,xcolor}
\usepackage{color}
\usepackage{bbm}
\usepackage{dsfont}
\usepackage{caption}
\usepackage{booktabs}
\usepackage{multirow}

\usepackage[colorlinks=true,      
			linkcolor=black,      
			citecolor=black,      
			filecolor=black,      
			urlcolor=black ]{hyperref}

\usepackage[font=small,labelfont=bf]{caption}

\usepackage{tikz}
\usetikzlibrary{shapes,matrix,decorations.shapes}
\usetikzlibrary{arrows,decorations.pathmorphing,backgrounds,positioning,fit,matrix}
\usetikzlibrary{shapes,arrows,chains}
\usetikzlibrary[calc]

\usepackage{algorithm}
\usepackage{algpseudocode}
\algrenewcommand\algorithmicrequire{\textbf{Input:}}
\algrenewcommand\algorithmicensure{\textbf{Output:}}

\newcounter{tempEquationCounter}
\newcounter{thisEquationNumber}

\newcolumntype{C}[1]{>{\centering}m{#1}}
\newtheorem{theorem}{Theorem}

\newtheorem{remark}{Remark}

\title{\LARGE \bf
Scalable analysis of linear networked systems via chordal decomposition
}

\author{Yang Zheng$^{\dagger}$, Maryam Kamgarpour$^{\ddagger}$, Aivar Sootla$^{\dagger}$ and Antonis Papachristodoulou$^{\dagger}$
\thanks{The work of Y. Zheng is supported in part by the Clarendon Scholarship, and in part by the Jason Hu Scholarship. The work of M. Kamgarpour is supported by ERC Starting Grant CONENE.}
\thanks{$^{\dagger} $Y. Zheng, A. Sootla and A. Papachristodoulou are with Department of Engineering Science at the University of Oxford. (E-mail:        \{yang.zheng, aivar.sootla, antonis\}@eng.ox.ac.uk)}%
\thanks{$^{\ddagger}$M. Kamgarpour is with Department of Electrical Engineering and Information Technology at ETH Zurich. (E-mail:        {mkamgar}@control.ee.ethz.ch)}%
}

\begin{document}

\maketitle
\thispagestyle{empty}
\pagestyle{empty}

\begin{abstract}
    This paper introduces a chordal decomposition approach for scalable analysis of linear networked systems, including stability, $\mathcal{H}_2$ and $\mathcal{H}_{\infty}$ performance. Our main strategy is to exploit any sparsity within these analysis problems and use chordal decomposition. We first show that Grone's and Agler's theorems can be generalized to block matrices with any partition. This facilitates networked systems analysis, allowing one to solely focus on the physical connections of networked systems to exploit scalability. Then, by choosing Lyapunov functions with appropriate sparsity patterns, we decompose large positive semidefinite constraints in all of the analysis problems into multiple smaller ones depending on the maximal cliques of the system graph. This makes the solutions more computationally efficient via a recent first-order algorithm. Numerical experiments demonstrate the efficiency and scalability of the proposed method.
\end{abstract}

\section{Introduction}

Large-scale networked systems consisting of multiple subsystems over a network have received considerable attention~\cite{siljak2011decentralized}. This class of systems arises in many practical applications, such as the smart gird~\cite{dor2014sparsity} and automated traffic systems~\cite{zheng2017distributed}. One of the challenges arising in these systems is to develop scalable methods that are able to solve the associated analysis and synthesis problems efficiently. However, classical methods often suffer from lack of scalability to large-scale systems, since computational demand usually grows rapidly as the system's dimension increases.

In the literature, there are two groups of scalable analysis techniques for large-scale networked systems: 1) \emph{compositional analysis}~\cite{moylan1978stability,meissen2015compositional,anderson2012decomposition}; and 2) \emph{positive systems theory}~\cite{farina2011positive,rantzer2015scalable,tanaka2011bounded,sootla2012scalable,sootla2017block}. The former method is usually carried out in the framework of dissipative systems, while the latter method aims to solve a special type of dynamical systems. The main strategy of compositional analysis is to find individual supply rates for each dissipative subsystem and then to establish a global storage function as a combination of the local storage functions~\cite{moylan1978stability}. Recently, Meissen \emph{et al.} employed a first-order method to optimize the local supply rates for certifying stability of an interconnected system~\cite{meissen2015compositional}, which might reduce the conservatism brought by individual storage functions. Anderson and Papachristodoulou proposed a decomposition technique based on graph partitioning that facilitates the compositional analysis~\cite{anderson2012decomposition}.  Another group of scalable strategies focuses on a particular class of systems, \emph{i.e.}, positive systems~\cite{farina2011positive}, where the system matrices only have nonnegative off-diagonal entries. It is well-known that stability and performance of positive systems can be verified using linear Lyapunov functions~\cite{rantzer2015scalable}, which can be computed by more scalable linear programs (LPs) instead of traditional semidefinite programs (SDPs). Tanaka and Langbort showed that it is necessary and sufficient to use a diagonal Lyapunov function in the KYP lemma for positive systems~\cite{tanaka2011bounded}. Sootla and Rantzer proposed scalable model reduction techniques for positive systems using linear energy functions~\cite{sootla2012scalable}.

In contrast to the compositional analysis and positive system theory, our approach focuses on the inherent structure and sparsity of networked systems and uses sparse optimization techniques, particularly chordal decomposition, to solve the analysis problems efficiently. This idea is in line with some of the early results in the field~\cite{rantzer2010distributed, mason2014chordal,andersen2014robust, ZMP2018Scalable}. Chordal decomposition is a celebrated result in linear algebra that connects sparse positive semidefinite matrices and chordal graphs. There is a broad literature regarding the applications of chordal graph properties in combinatorial problems, Cholesky factorization, matrix completion and sparse semidefinite optimization; see~\cite{vandenberghe2014chordal} for a comprehensive review.

In this paper, we introduce a chordal decomposition approach for scalable analysis of linear networked systems. We focus on the well-known convex formulations of the analysis problems, \emph{i.e.}, stability, $\mathcal{H}_2$ and $\mathcal{H}_{\infty}$ performance, and show how to decompose large positive semidefinite constraints in all of the analysis problems into multiple smaller ones, thus facilitating their solutions. Specifically, compared to~\cite{rantzer2010distributed, mason2014chordal,andersen2014robust, ZMP2018Scalable}, the contributions of this paper are: 1) we generalize Grone's and Agler's theorems to block matrices with arbitrary partition, and utilize the generalization for networked systems analysis; 2) we extend the scope of stability verification in~\cite{mason2014chordal} to $\mathcal{H}_2$ and $\mathcal{H}_{\infty}$ analyses of networked systems. Our approach can potentially be applied to other analysis and synthesis problems, such as structured model reduction and stabilizing feedback design.

The rest of this paper is organized as follows. In Section~\ref{Section:Preliminaries}, we present the problem formulation. Chordal decomposition in sparse SDPs is introduced in Section~\ref{Section:ChordalDecomposition}. Section~\ref{Section:ScalableAnalysis} presents the scalable analysis approach for stability, $\mathcal{H}_2$ and $\mathcal{H}_{\infty}$ performance. Numerical results are shown in Section~\ref{Section:Simulation}, and we conclude the paper in Section~\ref{Section:Conclusion}.
\vspace{-1mm}

\section{Preliminaries and Problem Statement} \label{Section:Preliminaries}
\vspace{-1mm}
\subsection{Preliminaries on graph-theoretic notions}
\vspace{-1mm}
A directed graph is denoted by $\mathcal{G}(\mathcal{V},\mathcal{E})$ and it consists of a set of nodes $\mathcal{V} = \{1,2, \ldots, n\}$ and a set of edges $\mathcal{E} \subseteq \mathcal{V} \times \mathcal{V}$. Graph $\mathcal{G}(\mathcal{V},\mathcal{E})$ is called undirected if $(u,v) \in \mathcal{E} \Leftrightarrow (v,u) \in \mathcal{E}$. A graph is \emph{complete} if there exists an edge between any pair of nodes. A clique is a subset of nodes $\mathcal{C} \subseteq \mathcal{V}$ that induces a complete subgraph. If $\mathcal{C}$ is not contained by any other clique, then it is referred to as a \emph{maximal clique}. A \emph{cycle} of length $k$ is a sequence $\{v_1, v_2, \ldots, v_k\} \subseteq \mathcal{V}$ with $(v_k, v_{1}) \in \mathcal{E}$ and $(v_i, v_{i+1}) \in \mathcal{E}, \forall i = 1, \ldots, k-1$. A \emph{chord} in a cycle $\{v_1, v_2, \ldots, v_k\}$ is an edge $(v_i,v_j)$ that connects two nonconsecutive nodes in the cycle.

An undirected graph is called \emph{chordal} if every cycle of length greater than three has a chord. Examples of chordal graphs include complete graphs and acyclic undirected graphs. Given a chordal graph, there are very efficient algorithms to find maximal cliques ~\cite{tarjan1984simple}. Non-chordal graphs $\mathcal{G}(\mathcal{V},\mathcal{E})$ can always be extended to a chordal graph $\hat{\mathcal{G}}(\mathcal{V},\hat{\mathcal{E}})$ by adding edges to $\mathcal{E}$, \emph{i.e.}, $\mathcal{E} \subset \hat{\mathcal{E}}$, and this process is called \emph{chordal extension}. There are several efficient heuristics to generate a good extension~\cite{vandenberghe2014chordal}. For example, the graph in Fig.~\ref{F:ChordalGraph}~(a) is non-chordal, and it can be chordally extended to that in Fig.~\ref{F:ChordalGraph}~(b) by adding an undirected edge $(3,4)$.  The graph in Fig.~\ref{F:ChordalGraph}~(b) has three maximal cliques: $\mathcal{C}_1 = \{1,3,4\}$, $\mathcal{C}_2 = \{2,3,5\}$ and $\mathcal{C}_3 = \{3,4,5\}$.


\begin{figure}
    \centering
    \setlength{\abovecaptionskip}{0em}
    \setlength{\belowcaptionskip}{-5pt}
    \begin{tikzpicture}
    \small
	  \matrix (m) [matrix of nodes,
	  		       row sep = 1.em,	
	  		       column sep = 1.4em,	
  			       nodes={circle, draw=black}] at (-2,0)
  		{ & 1 & \\ 2 &3 & 4 \\&5&\\};
		\draw (m-1-2) -- (m-2-2);
		\draw (m-1-2) -- (m-2-3);
		\draw (m-2-1) -- (m-2-2);
		\draw (m-2-1) -- (m-3-2);
		\draw (m-2-2) -- (m-3-2);
        \draw (m-2-3) -- (m-3-2);
		\node at (-2,-1.5) {(a)};
		\matrix (m2) [matrix of nodes,
	  		       row sep = 1.em,	
	  		       column sep = 1.4em,	
  			       nodes={circle, draw=black}] at (2,0)
        { & 1 & \\ 2 &3 & 4 \\&5&\\};
		\draw (m2-1-2) -- (m2-2-2);
		\draw (m2-1-2) -- (m2-2-3);
		\draw (m2-2-1) -- (m2-2-2);
		\draw (m2-2-1) -- (m2-3-2);
		\draw (m2-2-2) -- (m2-3-2);
        \draw (m2-2-2) -- (m2-2-3);
        \draw (m2-2-3) -- (m2-3-2);
		\node at (2,-1.5) {(b)};
	\end{tikzpicture}
    \caption{(a) Nonchordal graph: the cycle (1-3-5-4) is of length four but with no chords. (b) Chordal graph by adding an undirected edge $(3,4)$: all cycles of length greater than three have a chord.}
    \label{F:ChordalGraph}
\end{figure}
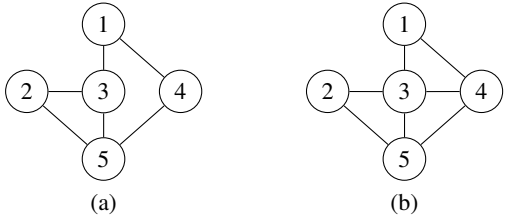

\subsection{Sparse block matrices}
A matrix $M \in \mathbb{R}^{N \times N}$ has \emph{$\alpha = \{\alpha_1, \alpha_2, \ldots, \alpha_n\}$-partitioning} with $N = \sum_{k=1}^n\alpha_i$ if
$$
    M = \begin{bmatrix} M_{11} & M_{12} & \ldots & M_{1n} \\
                        M_{12} & M_{22} & \ldots & M_{2n} \\
                        \vdots & \vdots & \ddots & \vdots \\
                        M_{n1} & M_{n2} & \ldots & M_{nn} \\  \end{bmatrix}
$$
where $M_{ij} \in \mathbb{R}^{\alpha_i \times \alpha_j}, i,j = 1, \ldots, n$. We describe the sparsity pattern of $\alpha$-partitioned $M$ by a graph $\mathcal{G}(\mathcal{V},\mathcal{E})$:
\begin{equation} \label{E:SparseMatrix}
    \mathbb{R}^{N\times N}_{\alpha}(\mathcal{E},0) = \{M \in \mathbb{R}^{N \times N} | M_{ij} = 0\; \text{if} \; (j,i) \notin \mathcal{E}^* \},
\end{equation}
where $\mathcal{E}^* = \mathcal{E} \cup \{(i,i), \forall i \in \mathcal{V}\}$. If $\mathcal{G}$ is undirected, we define the space of sparse symmetric block matrices as
\begin{equation} \label{E:SparseSymMatrix}
    \mathbb{S}^N_{\alpha}(\mathcal{E},0) = \{M \in \mathbb{S}^{N} | M_{ij} = M_{ji}^T= 0\; \text{if} \; (j,i) \notin \mathcal{E}^* \},
\end{equation}
and the cone of sparse block positive semidefinite (PSD) matrices as
$
    \mathbb{S}^N_{\alpha,+}(\mathcal{E},0) = \{M \in \mathbb{S}^N_{\alpha}(\mathcal{E},0) \mid M \succeq 0\}.
$
 Also, we define a cone $\mathbb{S}^N_{\alpha,+}(\mathcal{E},?)$ as the set of matrices in $\mathbb{S}^N_{\alpha}(\mathcal{E},0)$ that have a positive semidefinite completion, \emph{i.e.},
$
    \mathbb{S}^N_{\alpha,+}(\mathcal{E},?) = \mathbb{P}_{\mathbb{S}^N_{\alpha}(\mathcal{E},0)}(\mathbb{S}^N_+),
$ 
 where $\mathbb{P}$ denotes the projection onto the space of sparse matrices. 
\begin{remark} \label{R:BlockSparseMatrix}
    The definitions~\eqref{E:SparseMatrix} and~\eqref{E:SparseSymMatrix} also allow the block entry $M_{ij} = 0$ if $(j,i) \in \mathcal{E}^* $. Then we have $M \in \mathbb{S}^N_{\alpha}(\hat{\mathcal{E}},0)$ if $M \in \mathbb{S}^N_{\alpha}(\mathcal{E},0)$ and $\hat{\mathcal{E}}$ is a chordal extension of $\mathcal{E}$. This fact will be used in Section~\ref{Section:FOM}. The notation above is a natural extension of sparse scalar matrices to sparse block matrices with $\alpha$ partition. If $\alpha = \{1,1,\ldots,1\}$, then the notations are reduced to the normal cases, as used in~\cite{fukuda2001exploiting,vandenberghe2014chordal,ZFPGW2017chordal}.
\end{remark}

\subsection{Problem statement: analysis of linear networked systems}

We consider a network of linear heterogeneous subsystems interacting over a directed graph $\mathcal{G}(\mathcal{V},\mathcal{E})$. Each node in $\mathcal{V}$ represents a subsystem, and the edge $(i,j) \in \mathcal{E}$ means that subsystem $i$ exerts dynamical influence on subsystem $j$. The dynamics of subsystem $i \in \mathcal{V}$ are written as
\begin{equation} \label{E:SubsystemDynamics}
    \begin{aligned}
        \dot{x}_i(t) &= A_{ii}x_i(t) + \sum_{j \in \mathbb{N}_{i}}A_{ij}x_j(t) + B_i w_i(t), \\
        y_i(t) &= C_i x_i(t) + D_iw_i(t),
            \end{aligned}
\end{equation}
where $x_i \in \mathbb{R}^{\alpha_i}, y_i \in \mathbb{R}^{d_i}, w_i \in \mathbb{R}^{m_i}$ represent the local state, output and disturbance, respectively, and $\mathbb{N}_{i}$ denotes the neighbours of node $i$. By collecting the subsystems' states, the overall state-space model is then written concisely as
\begin{equation} \label{E:OverallDyanmics}
    \begin{aligned}
        \dot{x}(t) &= Ax(t) + B w(t), \\
        y(t) &= C x(t) + Dw(t),
            \end{aligned}
\end{equation}
where $x = [x_1^T, x_2^T, \ldots, x_n^T]^T$, and the vectors $y, w$ are defined similarly. The matrix $A$ is composed by blocks $A_{ij}$, and enjoys a block sparsity pattern $A \in \mathbb{R}_{\alpha}^{N\times N}(\mathcal{E},0)$. The matrices $B,C,D$ have block-diagonal structures.

In this paper, we develop scalable methods for three analysis problems~\cite{dullerud2013course} of the linear networked system~\eqref{E:OverallDyanmics}:

\emph{1) Stability analysis:} System~\eqref{E:OverallDyanmics} with $w = 0$ is asymptotically stable if and only if the Lyapunov linear matrix inequality (LMI) is feasible
\begin{equation} \label{E:StabilityCentralized}
    \begin{aligned}
    \text{find} \quad & P \succ 0 \\
        \text{subject to}\quad & A^TP+PA \prec 0. \\
    \end{aligned}
\end{equation}

\emph{2) $\mathcal{H}_2$ performance:} The $\mathcal{H}_2$ performance of a stable system~\eqref{E:OverallDyanmics} with $D=0$ can be computed as
\begin{equation}\label{E:H2Centralized}
    \begin{aligned}
    \min_{P} \quad & \text{Tr}(B^TPB) \\
        \text{subject to}\quad & A^TP+PA + C^TC \prec 0, \\
        & P \succ 0,
    \end{aligned}
\end{equation}
where $\|C(sI-A)^{-1}B\|_{\mathcal{H}_2} = \sqrt{\text{Tr}(B^TPB)}$.

\emph{3) $\mathcal{H}_{\infty}$ performance:} The $\mathcal{H}_{\infty}$ performance of a stable system~\eqref{E:OverallDyanmics} can be computed as
\begin{equation}\label{E:HinfCentralized}
    \begin{aligned}
    \min_{P} \quad & \gamma \\
        \text{subject to}\quad & \begin{bmatrix} A^TP+PA & PB & C^T \\B^TP & -\gamma I & D^T \\ C & D& -\gamma I \end{bmatrix} \prec 0, \\
        & P \succ 0,
    \end{aligned}
\end{equation}
where $\|C(sI-A)^{-1}B + D\|_{\mathcal{H}_{\infty}} = \gamma$.

Problems~\eqref{E:StabilityCentralized}-\eqref{E:HinfCentralized} are convex, and ready to be solved via existing interior-point solvers, such as SeDuMi~\cite{sturm1999using}. The main difficulty is that standard interior-point solvers suffer from scalability for large-scale problem instances. One major reason is that the constraints in~\eqref{E:StabilityCentralized}-\eqref{E:HinfCentralized} are imposed on the global system and consequently the computational complexity grows very quickly as the number of subsystems increases. Typically, the system graph $\mathcal{G}$ is sparse for practical large-scale systems, meaning that each subsystem only has physical connections with a few other subsystems. In this paper, we aim to exploit this sparsity in the algorithmic level to solve~\eqref{E:StabilityCentralized}-\eqref{E:HinfCentralized} efficiently. 

\begin{remark}
    Note that there are also other efficient formulations to test stability, and compute $\mathcal{H}_2$ and $\mathcal{H}_{\infty}$ performance~\cite{dullerud2013course}. An additional benefit of problems~\eqref{E:StabilityCentralized}-\eqref{E:HinfCentralized} is that we obtain an appropriate Lyapunov function $V(x) = x^TPx$. Also, problems~\eqref{E:StabilityCentralized}-\eqref{E:HinfCentralized} are helpful for some synthesis problems via a standard change of variables. In this paper, we will focus on~\eqref{E:StabilityCentralized}-\eqref{E:HinfCentralized}, and introduce a scalable approach to solve them efficiently when the system graph $\mathcal{G}$ is sparse.
\end{remark}

\section{Chordal Decomposition in Sparse SDPs} \label{Section:ChordalDecomposition}

In this section, we focus on the SDP formulations of the optimization problems~\eqref{E:StabilityCentralized}-\eqref{E:HinfCentralized} and explain how to decompose them using chordal graph theory. The standard \emph{primal form} of an SDP is
\begin{equation} \label{E:SDPprimal}
    \begin{aligned}
    \min_{X} \quad & \langle A_0,X \rangle \\
    \text{subject to} \quad & \langle A_i,X \rangle = b_i, i = 1, \ldots, m, \\
    & X \succeq 0,
    \end{aligned}
\end{equation}
and its standard \emph{dual form} is
\begin{equation} \label{E:SDPdual}
    \begin{aligned}
    \max_{y,Z} \quad & \langle b,y \rangle \\
    \text{subject to} \quad & Z+ \sum_{i=1}^m y_iA_i = A_0, \\
    & Z \succeq 0,
    \end{aligned}
\end{equation}
where $X$ is the primal variable, $y, Z$ are the dual variables, and $b \in \mathbb{R}^m, A_i \in \mathbb{S}^N, i = 0,1, \ldots, m$ are problem data.

\subsection{Chordal decomposition of sparse block PSD matrices}

Here, we introduce two key theorems that decompose $\mathbb{S}^N_{\alpha,+}(\mathcal{E},?)$ and $\mathbb{S}^N_{\alpha,+}(\mathcal{E},0)$ into a set of smaller and coupled cones, respectively. Given a partition $\alpha = \{\alpha_1,\alpha_2, \ldots, \alpha_n\}$ and a clique $\mathcal{C}_k$ of $\mathcal{G}$, we define an \emph{index matrix} $E_{\mathcal{C}_k} \in \mathbb{R}^{|\mathcal{C}_k| \times N}$ with $|\mathcal{C}_k| = \sum_{i \in \mathcal{C}_k}\alpha_i$ and $N = \sum_{i = 1}^n \alpha_i$ as
$$
    (E_{\mathcal{C}_k})_{ij} = \begin{cases} I_{\alpha_i}, \quad \text{if } \mathcal{C}_k(i) = j \\ 0, \qquad \text{otherwise} \end{cases}
$$
where $\mathcal{C}_k(i)$ denotes the $i$-th node in $\mathcal{C}_k$, sorted in the natural ordering. Given a block matrix $X \in \mathbb{S}^N$ with $\alpha$-partition, $E_{\mathcal{C}_k}XE_{\mathcal{C}_k}^T \in \mathbb{S}^{|\mathcal{C}_k|}$ extracts a principal submatrix defined by the clique $\mathcal{C}_k$, and the operation $E_{\mathcal{C}_k}^TYE_{\mathcal{C}_k}$ \emph{inflates} a $\vert\mathcal{C}_k\vert \times \vert \mathcal{C}_k\vert $ matrix into a sparse $N\times N$ matrix. Then, we have:
\begin{theorem} [Generalized Grone's theorem]\label{T:ChordalCompletionTheorem}
     Let $\mathcal{G}(\mathcal{V},\mathcal{E})$ be a chordal graph with a set of maximal cliques $\{\mathcal{C}_1,\mathcal{C}_2, \ldots, \mathcal{C}_p\}$. Given a partition $\alpha= \{\alpha_1,\alpha_2,\ldots,\alpha_n\}$ and $N = \sum_{i=1}^n \alpha_i$, then, $X\in\mathbb{S}^N_{\alpha,+}(\mathcal{E},?)$ if and only if
\end{theorem}
    \begin{equation} \label{E:DecompositionCompletion}
        E_{\mathcal{C}_k} X E_{\mathcal{C}_k}^T \in \mathbb{S}^{\vert \mathcal{C}_k \vert}_+,
    \quad k=1,\,\ldots,\,p.
\end{equation}

\begin{theorem} [Generalized Agler's theorem]\label{T:ChordalDecompositionTheorem}
     Let $\mathcal{G}(\mathcal{V},\mathcal{E})$ be a chordal graph with a set of maximal cliques $\{\mathcal{C}_1,\mathcal{C}_2, \ldots, \mathcal{C}_p\}$. Given a partition $\alpha= \{\alpha_1,\alpha_2,\ldots,\alpha_n\}$ and $N = \sum_{i=1}^n \alpha_i$, then, $Z\in\mathbb{S}^N_{\alpha,+}(\mathcal{E},0)$ if and only if there exist matrices $Z_k \in \mathbb{S}^{\vert \mathcal{C}_k \vert}_+$ for $k=1,\,\ldots,\,p$ such that
\end{theorem}
    \begin{equation} \label{E:DecompositionSparseCone}
    Z = \sum_{k=1}^{p} E_{\mathcal{C}_k}^T Z_k E_{\mathcal{C}_k}.
    \end{equation}

For the scalar case, \emph{i.e.}, $\alpha = \{1,1,\ldots,1\}$, Theorems~\ref{T:ChordalCompletionTheorem} and~\ref{T:ChordalDecompositionTheorem}  reduce to the Grone's~\cite{grone1984positive} and Agler's~\cite{agler1988positive} theorems, respectively. Here, we show that these two celebrated theorems can be generalized into sparse block matrices with an arbitrary partition. Due to lack of space, the proofs are omitted here\footnote{See proofs via \scriptsize \url{http://sysos.eng.ox.ac.uk/wiki/images/7/7c/ECC2018.pdf}}.
One direct benefit of the generalized Grone's and Agler's theorems is the convenience for networked system analysis, \emph{i.e.}, they allow us to solely focus on the physical connections characterized by $\mathcal{G}$ to exploit scalability. In this paper, \emph{chordal decomposition} refers to the application of Theorems~\ref{T:ChordalCompletionTheorem} and~\ref{T:ChordalDecompositionTheorem} to decompose a large sparse block PSD cone into a set of smaller but coupled PSD cones. In the next section, we summarize a recent first-order algorithm~\cite{ZFPGW2017chordal} that can exploit Theorems~\ref{T:ChordalCompletionTheorem} and~\ref{T:ChordalDecompositionTheorem}.

\subsection{Chordal decomposition in first-order methods} \label{Section:FOM}

Suppose the data matrices in~\eqref{E:SDPprimal} and \eqref{E:SDPdual} have an aggregate sparsity pattern:
$A_0, A_1, \ldots, A_m \in \mathbb{S}^N_{\alpha}(\mathcal{E},0).$ It is assumed that the pattern $\mathcal{E}$ is chordal (otherwise, a suitable chordal extension can be found; as stated in Remark~\ref{R:BlockSparseMatrix}, making an extension does not affect the problem data), with a set of maximal cliques $\mathcal{C}_1, \mathcal{C}_2, \ldots, \mathcal{C}_p$. Note that the cost function and equality constraints in~\eqref{E:SDPprimal} only depend on the entries $X_{ij}$ on its diagonal and $(i,j) \in \mathcal{E}$. The remaining elements simply guarantee that the matrix is PSD. Also, in~\eqref{E:SDPdual} any feasible solution $Z$ satisfies the sparsity pattern $\mathbb{S}^N_{\alpha}(\mathcal{E},0)$.  Recalling the definition of $\mathbb{S}^N_{\alpha,+}(\mathcal{E},?)$ and $\mathbb{S}^N_{\alpha,+}(\mathcal{E},0)$, and according to Theorems~\ref{T:ChordalCompletionTheorem} and~\ref{T:ChordalDecompositionTheorem}, we can equivalently reformulate the primal SDP~\eqref{E:SDPprimal} and the dual SDP~\eqref{E:SDPdual}, respectively, as
\begin{equation} \label{E:SDPprimalFirstOrder}
    \begin{aligned}
    \min_{X,X_k} \quad & \langle A_0,X \rangle \\
    \text{subject to} \quad & \langle A_i,X \rangle = b_i, & i = 1, \ldots, m, \\
    & X_k = E_{\mathcal{C}_k} XE_{\mathcal{C}_k}^T, & k = 1, \ldots, p, \\
    & X_k \in  \mathbb{S}^{\vert \mathcal{C}_k \vert}_+, &k=1,\,\ldots,\,p,
    \end{aligned}
\end{equation}
and
\begin{equation} \label{E:SDPdualFirstOrder}
    \begin{aligned}
    \max_{y,Z_k,V_k} \quad & \langle b,y \rangle \\
    \text{subject to} \quad  & \sum_{k=1}^p E_{\mathcal{C}_k}^T V_kE_{\mathcal{C}_k} + \sum_{i=1}^m y_iA_i = A_0, \\
    & Z_k = V_k, \;\;\,k = 1, \ldots, p, \\
    & Z_k \in  \mathbb{S}^{\vert \mathcal{C}_k \vert}_+, k=1,\,\ldots,\,p.
    \end{aligned}
\end{equation}
In~\eqref{E:SDPprimalFirstOrder} and~\eqref{E:SDPdualFirstOrder}, the original single large PSD cone has been replaced by multiple smaller PSD cones, coupled by a set of consensus variables. Then, first-order methods can be applied to the decomposed formulations~\eqref{E:SDPprimalFirstOrder} and~\eqref{E:SDPdualFirstOrder}~\cite{ZFPGW2016}, or their homogeneous self-dual embedding~\cite{ZFPGW2017chordal}, which lead to algorithms only involving parallel PSD projections onto $p$ smaller cones and a projection onto an affine set at each iteration. If the size of the largest maximal clique is small, then the reduction of cone dimensions enables us to compute PSD projections much more efficiently. Consequently, the application of first-order methods in the decomposed problems~\eqref{E:SDPprimalFirstOrder} and~\eqref{E:SDPdualFirstOrder} improves the scalability to solve sparse SDPs when seeking a solution of moderate accuracy. The interested reader is referred to~\cite{ZFPGW2016, ZFPGW2017chordal} for details. The MATLAB package CDCS~\cite{CDCS} provides an efficient implementation of the above decomposition method.

\begin{remark}
    Note that chordal sparsity has been identified as a key structure in large-scale sparse SDPs. In addition to the decomposition approach~\cite{ZFPGW2017chordal}, there are several other methods to exploit chordal properties (Theorems~\ref{T:ChordalCompletionTheorem} and~\ref{T:ChordalDecompositionTheorem}) in solving sparse SDPs, \emph{e.g.},~\cite{fukuda2001exploiting, andersen2010implementation}. In this paper, we use the decomposition approach~\cite{ZFPGW2017chordal} as an example to show the computational benefits brought by exploiting chordal sparsity in the analysis of large-scale networked systems. 
\end{remark}

\section{Scalable Performance Analysis of Sparse Systems} \label{Section:ScalableAnalysis}

This section applies the chordal decomposition techniques in stability, $\mathcal{H}_2$ and $\mathcal{H}_{\infty}$ analyses of linear networked systems. Our strategy is to restrict the sparsity pattern of $P$, such that the sparsity structure of the dynamical system is preserved in the SDP formulations of~\eqref{E:StabilityCentralized}-\eqref{E:HinfCentralized}. This allows one to decompose a single large PSD constraint into multiple smaller ones using chordal decomposition, thus facilitating their solutions using sparse optimization techniques~\cite{CDCS}. We note that the proposed method of this section may introduce certain conservatism for general networked systems since we use a sparse Lyapunov function.

\subsection{Stability verification} \label{subSection:stability}

We first show that the sparsity pattern of $A^TP+PA$ reflects the aggregate sparsity pattern of the resulting SDP formulation. Let us write the Lyapunov LMI as
\begin{equation}~\label{E:LyapLMI}
    \begin{bmatrix} -P & 0 \\0 & A^TP+PA\end{bmatrix} \prec 0.
\end{equation}
There are up to $m = \frac{N(N+1)}{2}$ free variables in matrix $P$. We denote $W_1, W_2, \ldots, W_m$ as the standard basis matrices for $\mathbb{S}^N$, and define the matrices $A_1, A_2, \ldots, A_m \in \mathbb{S}^{2N}$ as
\begin{equation} \label{E:StabilityBasis}
    A_i = \begin{bmatrix} - W_i & 0 \\0 & A^TW_i + W_iA \end{bmatrix}, i = 1, 2, \ldots, m.
\end{equation}
Then,~\eqref{E:LyapLMI} can be reformulated into a standard dual SDP
    \begin{equation} \label{E:SDPdualStability}
    \begin{aligned}
    \max_{y,Z} \quad & \langle b,y \rangle \\
    \text{subject to} \quad & Z+ \sum_{i=1}^m y_iA_i = A_0, \\
    & Z \succeq 0,
    \end{aligned}
\end{equation}
where $y \in \mathbb{R}^m, Z \in \mathbb{S}^{2N}_+$, $A_0= -\epsilon I, \epsilon>0, b = 0$, and $A_i$ is defined in~\eqref{E:StabilityBasis}. At this point, we know that the aggregate sparsity pattern of~\eqref{E:SDPdualStability} is
$$
    \mathcal{P}(A_0) \cup \mathcal{P}(A_1) \cup \cdots \cup \mathcal{P}(A_m) = \mathcal{P}\left(\begin{bmatrix} -P & 0 \\0 & A^TP+PA\end{bmatrix}\right),
$$
where $\mathcal{P}(\cdot)$ denotes the sparsity pattern of a matrix. The aggregate sparsity pattern of~\eqref{E:SDPdualStability} directly depends on the sparsity patterns of $P$ and $A^TP + PA$.

For a networked system~\eqref{E:OverallDyanmics}, matrix $A$ has an inherent structure described by $\mathcal{G}(\mathcal{V},\mathcal{E})$, \emph{i.e.}, $A \in \mathbb{R}_{\alpha}^{N \times N}(\mathcal{E},0)$, where $\alpha = \{\alpha_1, \alpha_2, \ldots, \alpha_n\}$ denotes the dimensions of local states. Apparently, a dense $P$ has no conservatism in certifying stability, but leads to a full pattern of $A^TP+PA$. To preserve the sparsity structure $\mathcal{G}$, we consider the following problem.
\begin{description}
  \item[Q$^*$] Given a sparse $A \in \mathbb{R}^{N \times N}_{\alpha}(\mathcal{E},0)$, find a sparsity pattern of $P$, such that the sparsity pattern of $A^TP+PA$ inherits the original pattern $\mathcal{E}$ (more favorably, the resulting pattern is chordal with small maximal cliques).
\end{description}

We note that a complete answer to Q$^*$ is difficult for general systems, especially considering the relationship between sparsity (\emph{i.e.}, efficiency) and conservativeness. One trivial choice is a block-diagonal $P$ with block sizes compatible with the subsystem sizes $\alpha_i$. Then, the graph structure in the dynamical system~\eqref{E:SubsystemDynamics} is naturally inherited in~\eqref{E:SDPdualStability}, \emph{i.e.},
\begin{equation}
    A^TP+PA \in \mathbb{S}^{N}_{\alpha}(\mathcal{E}\cup\mathcal{E}_r,0),
\end{equation}
where $\mathcal{E}_r$ denotes a set of reverse edges obtained by reversing the order of the pairs in $\mathcal{E}$. We note that the existence of block-diagonal $P$ is investigated in~\cite{sootla2017block}, and diagonal $P$ is necessary and sufficient to certify stability of positive systems~\cite{rantzer2015scalable}. Other choices of $P$ are available for special graphs $\mathcal{G}$, such as trees and cycles~\cite{mason2014chordal}.

In this paper, we assume the pattern of $A^TP+PA$ can be described by a chordal graph $\mathcal{G}_c(\mathcal{V},\mathcal{E}_c)$ with a set of maximal cliques $\mathcal{C}_1, \ldots, \mathcal{C}_p$. As mentioned above, one basic choice is a block-diagonal $P$. If $\mathcal{E}\cup\mathcal{E}_r$ is non-chordal, then we can make a chordal extension to get $\mathcal{E}_c$. In~\eqref{E:SDPdualStability}, the single large PSD cone $Z \succeq 0$ has two blocks, where the upper-left one corresponds to block-diagonal $P$ and the bottom-right one can be replaced by $\mathbb{S}^N_{\alpha,+}(\mathcal{E}_c,0)$. Then, $\mathbb{S}^N_{\alpha,+}(\mathcal{E}_c,0)$ can be subsequently decomposed into multiple smaller cones using chordal decomposition (see Theorem 1 and the reformulations~\eqref{E:SDPprimalFirstOrder} and~\eqref{E:SDPdualFirstOrder}). Consequently, if the largest maximal clique is small, the SDP formulation~\eqref{E:SDPdualStability} can be expected to be solved efficiently for sparse systems using the first-order method described in Section~\ref{Section:FOM}.

\subsection{$\mathcal{H}_2$ performance}

Similar to the stability analysis, the $\mathcal{H}_2$ optimization problem \eqref{E:H2Centralized} can be reformulated into a standard SDP of primal form~\eqref{E:SDPprimal} or dual form~\eqref{E:SDPdual}. The aggregate sparsity pattern of the resulting SDP is determined by the pattern of
$
    A^TP+PA + C^TC.
$
Considering the structure of networked system~\eqref{E:SubsystemDynamics}, we have
\begin{equation}
    \mathcal{P}(A^TP+PA + C^TC) = \mathcal{P}(A^TP+PA).
\end{equation}
Then, the argument for stability analysis can be applied to $\mathcal{H}_2$ analysis for the purpose of scalable computation. We assume the pattern of $A^TP+PA$ can be described by a chordal graph $\mathcal{G}_c(\mathcal{V},\mathcal{E}_c)$, leading to
$
    A^TP+PA + C^TC \in \mathbb{S}^{N}_{\alpha}(\mathcal{E}_c,0).
$
Consequently, the sparse optimization technique~\cite{CDCS} is ready to solve the decomposed version of~\eqref{E:H2Centralized}. Note that we can only obtain an approximated (upper bound) $\mathcal{H}_2$ performance in general due to using a sparse $P$.

\subsection{$\mathcal{H}_{\infty}$ performance}

When reformulating the $\mathcal{H}_{\infty}$ analysis problem~\eqref{E:HinfCentralized} into a standard SDP, the aggregate sparsity pattern depends on the pattern of the following matrix
\begin{equation} \label{E:PatternHinf}
    M = \begin{bmatrix} A^TP+PA & PB & C^T \\B^TP & -\gamma I & D^T \\ C & D& -\gamma I \end{bmatrix}.
\end{equation}
According to the inherent structure in~\eqref{E:SubsystemDynamics}, we know $A \in \mathbb{R}_{\alpha}^{N \times N}(\mathcal{E},0)$ and $B,C,D$ are block-diagonal. If we restrict $P$ to be block-diagonal with compatible block sizes, then the entry $PB$ is also block-diagonal. The sparsity pattern of the first block on the diagonal is
$
    A^TP+PA \in \mathbb{S}^{N}_{\alpha}(\mathcal{E}_c,0),
$
where $\mathcal{E}_c$ is the chordal extension of $\mathcal{E}\cup\mathcal{E}_r$, defined in Section~\ref{subSection:stability}. Then, we have the following result.

\begin{theorem} \label{T:HinfPattern}
    Consider a networked system with dynamics~\eqref{E:OverallDyanmics} and a block-diagonal $P$. Suppose that the sparsity pattern of $A^TP+PA$ has $p$ maximal cliques $\mathcal{C}_1,\mathcal{C}_2,\ldots,\mathcal{C}_p$, and the cardinality of the largest maximal clique is $h$. Then,
    \begin{enumerate}
      \item the block matrix $M$ in~\eqref{E:PatternHinf} has a partition $\hat{\alpha} = \{\alpha_1,\alpha_2,\ldots,\alpha_n,m_1,m_2,\ldots,m_n,d_1,d_2,\ldots,d_n\}$;
      \item the sparsity pattern of $M$, denoted as $M \in \mathbb{S}^{\hat{N}}_{\hat{\alpha}}(\hat{\mathcal{E}},0)$, has $p+n$ maximal cliques, and the cardinality of the largest maximal clique is $\max\{h,3\}$.
    \end{enumerate}
\end{theorem}

\begin{proof}
    According to~\eqref{E:SubsystemDynamics}, it is straightforward to see that the block matrix $M$ in~\eqref{E:PatternHinf} has a partition $\hat{\alpha} = \{\alpha_1,\alpha_2,\ldots,\alpha_n,m_1,m_2,\ldots,m_n,d_1,d_2,\ldots,d_n\}$, where $\alpha_i,m_i,d_i (i = 1, \ldots, n)$ are the dimensions of local states, disturbances and outputs, respectively.

    Let us first consider the following block
    \begin{equation} \label{E:SubmatrixM}
        M_1 = \begin{bmatrix} A^TP+PA & PB \\B^TP & -\gamma I \end{bmatrix} \in \mathbb{S}^{\hat{N}_1}_{\hat{\alpha}_1}(\hat{\mathcal{E}}_1,0),
    \end{equation}
    where the partition $\hat{\alpha}_1 = \{\alpha_1,\alpha_2,\ldots,\alpha_n,m_1,m_2,\ldots,m_n\}$ and $\hat{N}_1 = \sum_{i=1}^n(\alpha_i+m_i)$. Since $PB$ and $\gamma I$ are block diagonal, every node $i \in \{n+1,\ldots, n+n\}$ is only connected to one node $i-n$. Then, the edge set for $M_1$ is shown as
    \begin{equation} \label{E:HinfEdge1}
        \hat{\mathcal{E}}_1 = {\mathcal{E}}_c \bigcup \left\{(i,i+n)\mid i = 1, \ldots, n\right\},
    \end{equation}
    indicating that the maximal cliques of $\hat{\mathcal{E}}_1$ are given by
    \begin{equation}
        \mathcal{C}_1,\ldots ,\mathcal{C}_p, \mathcal{C}_{p+i} = \{i,i+n\}, i = 1, \ldots, n.
    \end{equation}

    Next, according to~\eqref{E:PatternHinf} and \eqref{E:SubmatrixM}, we know
    \begin{equation}
        M = \begin{bmatrix} M_1 & H \\H^T & -\gamma I \end{bmatrix} \in \mathbb{S}^{\hat{N}}_{\hat{\alpha}}(\hat{\mathcal{E}},0),
    \end{equation}
    where $H^T = \begin{bmatrix}C&D\end{bmatrix}$ and $\hat{N} = \sum_{i=1}^n(\alpha_i+m_i+d_i)$. Since the matrices $C, D$ and $\gamma I$ are block diagonal, every node $i \in \{2n+1,\ldots, 2n+n\}$ is connected to another two nodes $i-n, i-2n$. Consequently, the edge set for $M$ is given by
    \begin{equation*}
        \hat{\mathcal{E}} = \hat{\mathcal{E}}_1 \bigcup \left\{(i,i+2n),(i+n,i+2n)\mid i = 1, \ldots, n\right\}.
    \end{equation*}
    According to the edge set $\hat{\mathcal{E}}_1$~\eqref{E:HinfEdge1}, we know that in the edge set $\hat{\mathcal{E}}$, $\{i,i+n,i+2n\}$ forms a maximal clique. This implies that the maximal cliques of $\hat{\mathcal{E}}$ are
    \begin{equation}
        \mathcal{C}_1,\ldots ,\mathcal{C}_p, \mathcal{C}_{p+i} = \{i,i+n,i+2n\}, i = 1, \ldots, n.
    \end{equation}

    Therefore, $\mathbb{S}^{\hat{N}}_{\hat{\alpha}}(\hat{\mathcal{E}},0)$ has $p+n$ maximal cliques, and the cardinality of the largest maximal clique is $\max\{h,3\}$.
\end{proof}

Although $\mathcal{H}_{\infty}$ analysis problem~\eqref{E:HinfCentralized} appears to be more complex than the Lyapunov LMI~\eqref{E:StabilityCentralized}, Theorem~\ref{T:HinfPattern} shows that the underlying maximal cliques are similar and that the cardinality of the largest maximal clique for~\eqref{E:HinfCentralized} and~\eqref{E:StabilityCentralized} is almost identical. Therefore, the strategy for stability analysis can be applied to $\mathcal{H}_{\infty}$ analysis problem~\eqref{E:HinfCentralized}: the single large PSD cone can be decomposed into $p+n$ smaller ones, and the sparse optimization technique~\cite{CDCS} can be used to solve the decomposed problem in a scalable fashion. 

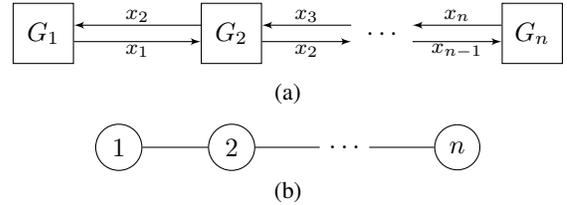
\begin{figure}[t]
    \centering
    \setlength{\abovecaptionskip}{0em}
    \setlength{\belowcaptionskip}{0em}
    \begin{tikzpicture}[every node/.style={minimum width=0.8cm, minimum height=0.8cm,text centered,align=center}]
        \node[draw,rectangle] (B) at (0,0) {$G_1$};
        \node[draw,rectangle] (C) at (2.5,0) {$G_2$};
        \node (D) at (4.5,0) {$\ldots$};
        \node[draw,rectangle] (E) at (6.5,0) {$G_n$};
        \draw[-latex'] (B.345) -- node[rotate=0,below=-0.8em,font=\footnotesize] {$x_1$} (C.195);
        \draw[latex'-] (B.15) -- node[rotate=0,above=-0.8em,font=\footnotesize] {$x_2$} (C.165);
        \draw[-latex'] (C.345) -- node[below=-0.8em,font=\footnotesize] {$x_2$}(D.195);
        \draw[latex'-] (C.15) -- node[rotate=0,above=-0.8em,font=\footnotesize] {$x_3$} (D.165);
        \draw[-latex'] (D.345) -- node[rotate=0,below=-0.8em,font=\footnotesize] {$x_{n-1}$} (E.195);
        \draw[latex'-] (D.15) -- node[rotate=0,above=-0.8em,font=\footnotesize] {$x_{n}$} (E.165);
        \node at (3.25,-0.8) {\small (a)};
    \end{tikzpicture}
    \begin{tikzpicture}[every node/.style={minimum width=0.6cm, minimum height=0.6cm,text centered,align=center}]
        \node[draw,circle] (B) at (0,0) {$1$};
        \node[draw,circle] (C) at (1.5,0) {$2$};
        \node (D) at (3,0) {$\ldots$};
        \node[draw,circle] (E) at (4.5,0) {$n$};
        \draw (B) --  (C);
        \draw (C) -- (D);
        \draw (D) -- (E);
        \node at (2.25,-0.6) {\small (b)};
    \end{tikzpicture}
    \caption{ (a) A chain of $n$ subsystems; (b) a simplified line graph.}  \label{F:Chains}
\end{figure}

\begin{figure}[t]
  \centering
    \setlength{\belowcaptionskip}{-8pt}
  \includegraphics[scale = 0.76,angle =0]{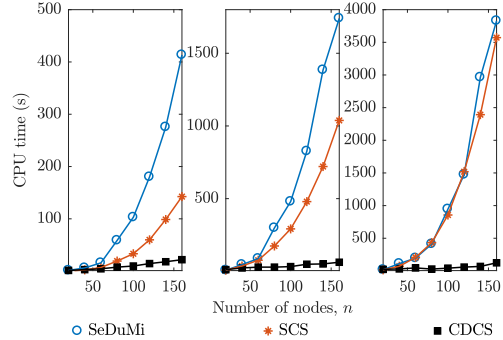}\\
  \caption{CPU time (s) required by SeDuMi, SCS and CDCS to solve the SDPs of the analysis problems~\eqref{E:StabilityCentralized}-\eqref{E:HinfCentralized} of a chain of subsystems.}
  \label{F:CPUlinegraph}
\end{figure}

\section{Numerical Simulations} \label{Section:Simulation}

To show the efficiency of the chordal decomposition approach, we consider a chain of $n$ subsystems where each subsystem has physical interactions with its two neighbouring ones except the first and last subsystem, which only interacts with one neighbouring subsystem; see Fig.~\ref{F:Chains}~(a) for illustration. A simplified version of this chain is shown in Fig.~\ref{F:Chains}~(b). In this case, the maximal cliques of the graph in Fig.~\ref{F:Chains}~(b) are $\{i,i+1\}, i = 1, \ldots, n-1$, and the cardinality of the largest maximal clique is only 2. 

\begin{table}[t]
    \centering
    \scriptsize
    \setlength{\abovecaptionskip}{0em}
    \renewcommand\arraystretch{0.9}
    \caption{Approximated $\mathcal{H}_2$ and $\mathcal{H}_{\infty}$ performance of a chain of subsystems computed by different solvers.}
    \label{T:CostChain}
    \begin{tabular}{ c C{4mm} C{5mm} C{5mm} C{5mm} C{0.4mm} C{4mm} C{5mm} C{5mm} c}
        \hline \toprule[1pt] 
        & \multicolumn{4}{c}{$\mathcal{H}_2$} & & \multicolumn{4}{c}{$\mathcal{H}_{\infty}$}  \\
          \cline{2-5} \cline{7-10} \\[-0.75em]
        $n$  & $\dagger$ & sedumi &  SCS  & CDCS &    & $\ddagger$ & sedumi  & SCS  & CDCS \\
         \cline{2-5} \cline{7-10}\\[-0.5em]
        $20$  &~9.70& 17.73 &   17.73 &    17.73  &  &3.65& 3.66 &  3.70&    3.66\\
        $40$  &11.66& 20.07 &   20.07  &   20.07 &  &3.67& 3.68 &  3.74&       3.68\\
        $60$  &14.72& 25.78 &   25.79 &     25.78  &&3.75&  3.77&    3.85&       3.77\\
        $80$  &16.85& 28.70 &   28.71  &  28.69 &&4.32& 4.34  &  4.37 & 4.34\\
        $100$ &18.08&  29.88 &   29.91 &     29.88 &&3.91& 3.92&    3.96&       3.92\\
        $120$ &19.71& 32.10 &   32.12  &  32.09 &&4.02&4.03&    4.10  &   4.04 \\
        $140$ &21.51&  35.59 &   35.64 &     35.58 &&4.09& 4.10&    4.16&       4.11\\
        $160$ &24.64& 40.65  & 40.73     &40.65 & &4.07& 4.08 & 4.18&  4.09\\
        \bottomrule[1pt]
        \end{tabular}
        \scriptsize
\raggedright
$\dagger$: Accurate $\mathcal{H}_2$ performance returned by the MATLAB routine {\tt  norm(sys,2)}.\newline
$\ddagger$: Accurate $\mathcal{H}_{\infty}$ performance returned by the MATLAB routine {\tt  norm(sys,inf)}.\newline
\vspace{-5 mm}
\end{table}

We solved the SDP formulations of stability analysis~\eqref{E:StabilityCentralized}, $\mathcal{H}_2$ performance~\eqref{E:H2Centralized}, and $\mathcal{H}_{\infty}$ performance~\eqref{E:HinfCentralized} using standard dense solvers: SeDuMi~\cite{sturm1999using} and SCS~\cite{scs}, as well as using the sparse conic solver CDCS~\cite{CDCS} that exploits chordal sparsity. Block-diagonal $P$ was used in the formulations. 
For the interior-point solver SeDuMi, we used its default parameters, and the first-order solvers SCS and CDCS were called with termination tolerance $10^{-4}$ and number of iterations limited to $2000$. All simulations were run on a PC with a 2.8 GHz Intel Core i7 CPU and 8GB of RAM. In the simulations, the state dimension $n_i$ was chosen randomly from 5 to 10, and the dimensions of output and disturbance $(d_i, m_i)$ were chosen randomly from 1 to 5. Then, we generated random matrices $A_{ii}, A_{ij}, B_i, D_i$ and imposed the global state matrix $A$ with negative eigenvalues by setting $A := A - (\lambda_{\max}+5)I$, where $\lambda_{\max}$ denotes the maximum real part of the eigenvalues of $A$.

Fig.~\ref{F:CPUlinegraph} shows the CPU time in seconds required by the solvers for testing stability, and computing approximated $\mathcal{H}_2$ and $\mathcal{H}_{\infty}$ performance. The chordal decomposition approach (via CDCS) took significantly less time than standard dense methods (using either SeDuMi or SCS). Moreover, the CPU time required by CDCS seems to grow linearly as the system size increases. This is expected since the size of the largest maximal clique is fixed for a line graph, indicating that the size of the PSD cones after decomposition is fixed and only the number of PSD cones increases linearly as growth of the graph size.
Finally, Table~\ref{T:CostChain} lists the $\mathcal{H}_2$ and $\mathcal{H}_{\infty}$ performance computed by different solvers. We can see that using block-diagonal $P$ indeed brought certain conservatism when searching for performance bounds.

\balance
\section{Conclusion} \label{Section:Conclusion}

In this paper, we have introduced scalable analysis techniques for sparse linear networked systems by exploiting chordal decomposition and using a recent first-order algorithm. The numerical results have shown that when the largest maximal clique is small, the chordal decomposition approach is significantly faster than the standard dense method. This makes it a promising approach for large sparse systems analysis. Future work will consider non-block diagonal Lyapunov functions that can preserve the sparsity pattern in the analysis problems. 
Also, there are several interior-point methods that are able to exploit chordal properties in solving sparse SDPs~\cite{fukuda2001exploiting, andersen2010implementation}. It will be interesting to apply these solvers in sparse systems analysis and synthesis as well.

\bibliographystyle{IEEEtran}
\bibliography{Reference}

\end{document}